\documentclass[12pt,reqno]{amsart}
\usepackage{amsmath,amssymb,amsthm,mathtools,calc,verbatim,enumitem,tikz,url,hyperref,mathrsfs,cite,fullpage}
\usepackage{bbm}
\usepackage{stmaryrd}
\usepackage{textcomp}
\usepackage{setspace}
\usepackage{amssymb}
\usepackage{amsthm}
\usepackage{amsmath}
\usepackage{graphicx}
\usepackage{marvosym}
\usepackage{empheq}
\usepackage{latexsym}
\usepackage{fontenc}
\usepackage{color}

\renewcommand\le{\leqslant}
\renewcommand\ge{\geqslant}
\renewcommand\to{\rightarrow}

	\def\<{\langle }
	\def\>{\rangle }

\usepackage{amssymb}
\newtheorem{theor}{Theorem}[section]
\newtheorem{prop}[theor]{Proposition}
\newtheorem{lemma}[theor]{Lemma}
\newtheorem{coro}[theor]{Corollary}
\newtheorem{conj}[theor]{Conjecture}
\theoremstyle{definition}
\newtheorem{defi}[theor]{Definition}
\theoremstyle{remark}
\newtheorem{rema}[theor]{Remark}
\newtheorem{exam}[theor]{Example}
\newtheorem{nota}[theor]{Notation}

\newtheorem*{claim*}{Claim}
\newtheorem*{qu*}{Question}

\newcommand{\genA}{[A]}
\newcommand{\aab}{m}
\newcommand{\Zdd}{\mathbb Z^2}
\newcommand{\Zdt}{\mathbb Z^3}
\newcommand{\Zd}{\mathbb Z^d}

\newcommand{\pp}{\mathbb P}
\newcommand{\p}{\mathbb P}

\newcommand{\U}{\mathcal U}
\newcommand{\UU}{\mathcal U}
\newcommand{\dhp}{\mathbb H_u}
\newcommand{\Ss}{\mathcal{S}}

\usepackage{dsfont}
\usepackage{caption}
\usepackage{subcaption}
\usepackage{pifont}
\usepackage{anysize}
\marginsize{1in}{1in}{1in}{1in}
\begin{document}

\title{Three-dimensional 2-critical bootstrap percolation: The stable sets approach}
\author{Daniel Blanquicett}

\address{Mathematics Department,
University of California, Davis, CA 95616, USA}
\email{drbt@math.ucdavis.edu}
\thanks{{\it Date}: January 26, 2022.
}
	
\begin{abstract}
Consider a $p$-random subset $A$ of initially infected vertices in the discrete cube $[L]^3$,
and assume that the neighbourhood of each vertex consists of the $a_i$ nearest neighbours
in the $\pm e_i$-directions for each
$i \in \{1,2,3\}$, where $a_1\le a_2\le a_3$.
Suppose we infect any healthy vertex $v\in [L]^3$ already having $r$ infected neighbours,
and that infected sites remain infected forever.
In this paper we determine $\log$ of the critical length for percolation up to a constant factor, 
for all $r\in \{a_3+1,
\dots, a_3+a_2\}$ with $a_3\ge a_1+a_2$. 
We moreover give upper bounds for all remaining cases $a_3 < a_1+a_2$ 
and believe that they are tight up to a constant factor.
\end{abstract}
	
\maketitle 
\section{Introduction}
The study of bootstrap processes on graphs was initiated in 1979 by Chalupa,
Leath and Reich~\cite{ChLR79}, and is motivated by problems arising from statistical physics, such as the Glauber dynamics of 
the zero-temperature Ising model, and kinetically constrained spin models of the liquid-glass transition 
(see, e.g.,~\cite{DB21,FSS02,Morris09,MMT18,Morris17}). 
The $r$-neighbour bootstrap process on a locally finite graph $G$ is a monotone cellular automata on the 
configuration space $\{0,1\}^{V(G)}$, (we call vertices in state $1$ ``infected"), evolving in discrete time
in the following way: $0$ becomes $1$ when it has at least $r$ neighbours in state $1$, and infected vertices remain infected forever.
Throughout this paper, $A$ denotes the initially infected set, and we write $\genA=G$
if the state of each vertex is eventually 1.

We will focus on \emph{anisotropic} bootstrap models, which are $d$-dimensional analogues of a family of
(two-dimensional) processes studied by Duminil-Copin, van Enter and Hulshof \cite{EH07,DCE13,DEH18}.
In these models the graph $G$ has
vertex set $[L]^d$, and the neighbourhood of each vertex consists of the $a_i$ nearest neighbours in the
$-e_i$ and $e_i$-directions for each $i \in [d]$,
where $a_1\le \cdots\le a_d$ and $e_i\in\Zd$ denotes the $i$-th canonical unit vector.
In other words, $u,v\in [L]^d$ are neighbours if (see Figure \ref{figanis3d} for $d=3$)
\begin{align}\label{neigh3} 
u-v\in N_{a_1,\dots,a_d}:=\{\pm e_1,\dots, \pm a_1e_1\}\cup \cdots \cup \{\pm e_d,\dots, \pm a_de_d\}.
\end{align}
We also call this process the $\mathcal N_r^{a_1,\dots, a_d}$-{\it model}.
Our initially infected set $A$ 
is chosen according to the Bernoulli product measure $\p_p=\bigotimes_{v\in [L]^d}$Ber$(p)$,
and we are interested in the so-called {\it critical length for percolation},
for small values of $p$
\begin{equation}\label{criticalL}
 L_c(\mathcal N_r^{a_1,\dots,a_d},p):= \min\{L\in\mathbbm N: \pp_p(\genA=[L]^d
 )\ge 1/2\}.
 \end{equation}

 
The analysis of these bootstrap processes for $a_1=\cdots= a_d=1$ was initiated by Aizenman and Lebowitz~\cite{AL88} in 1988,
who determined the magnitude of the critical length 
up to a constant factor in the exponent for the $\mathcal N_2^{1,\dots,1}$-model (in other words, they determined the
`metastability threshold' for percolation). In the case $d = 2$, Holroyd~\cite{H03} determined (asymptotically, as $p \to 0$) the constant in the exponent  (this is usually called a sharp metastability threshold).

For the general $\mathcal N_r^{1,\dots,1}$-model with $2\le r\le d$, the threshold was determined by Cerf and Cirillo \cite{CC99} and Cerf
and Manzo \cite{CM02}, and the sharp threshold by Balogh, Bollob\'as and Morris \cite{BBM09}
and Balogh, Bollob\'as, Duminil-Copin and Morris  \cite{BBDM12}: for all $d\ge r\ge 2$ there exists a computable constant 
$\lambda(d,r)$ such that, as $p\to 0$,
\begin{equation*}
 L_c(\mathcal N_r^{1,\dots,1},p) = \exp_{(r-1)}\bigg(\frac{\lambda(d,r) + o(1)}{p^{1/(d-r+1)}}\bigg).
\end{equation*}

In dimension $d=2$, we write $a_1=a, a_2=b$, and the $\mathcal N_r^{a,b}$-model is called isotropic when $a=b$ and anisotropic when $a<b$.
Hulshof and van Enter \cite{EH07} determined the threshold for the first interesting anisotropic model given by the family $\mathcal N_{3}^{1,2}$, and the 
corresponding sharp threshold was determined by Duminil-Copin and van Enter  \cite{DCE13}.

The threshold was also determined in the general case $r=a+b$ by van Enter and Fey 
 \cite{AA12} and the proof can be extended to all $b+1\le r\le a+b$: as $p\to 0$,
\begin{equation}\label{paso1}
\log L_c\left(\mathcal N_{r}^{a,b},p\right)=
\begin{cases}
\Theta\left(p^{-(r-b)}\right) & \textup{if }b=a,\\
\Theta\left(p^{-(r-b)}(\log p)^2\right) & \textup{if }b>a.
\end{cases}
\end{equation}

 
\subsection{Anisotropic bootstrap percolation on $[L]^3$}
In this paper we consider the three-dimensional analogue of the anisotropic bootstrap process studied by Duminil-Copin,
van Enter and Hulshof. 
In dimension $d=3$, we write $a_1=a, a_2=b$ and $a_3=c$.
\vskip -.2cm
\begin{figure}[ht]
	\centering
	\includegraphics[width=0.35\textwidth]{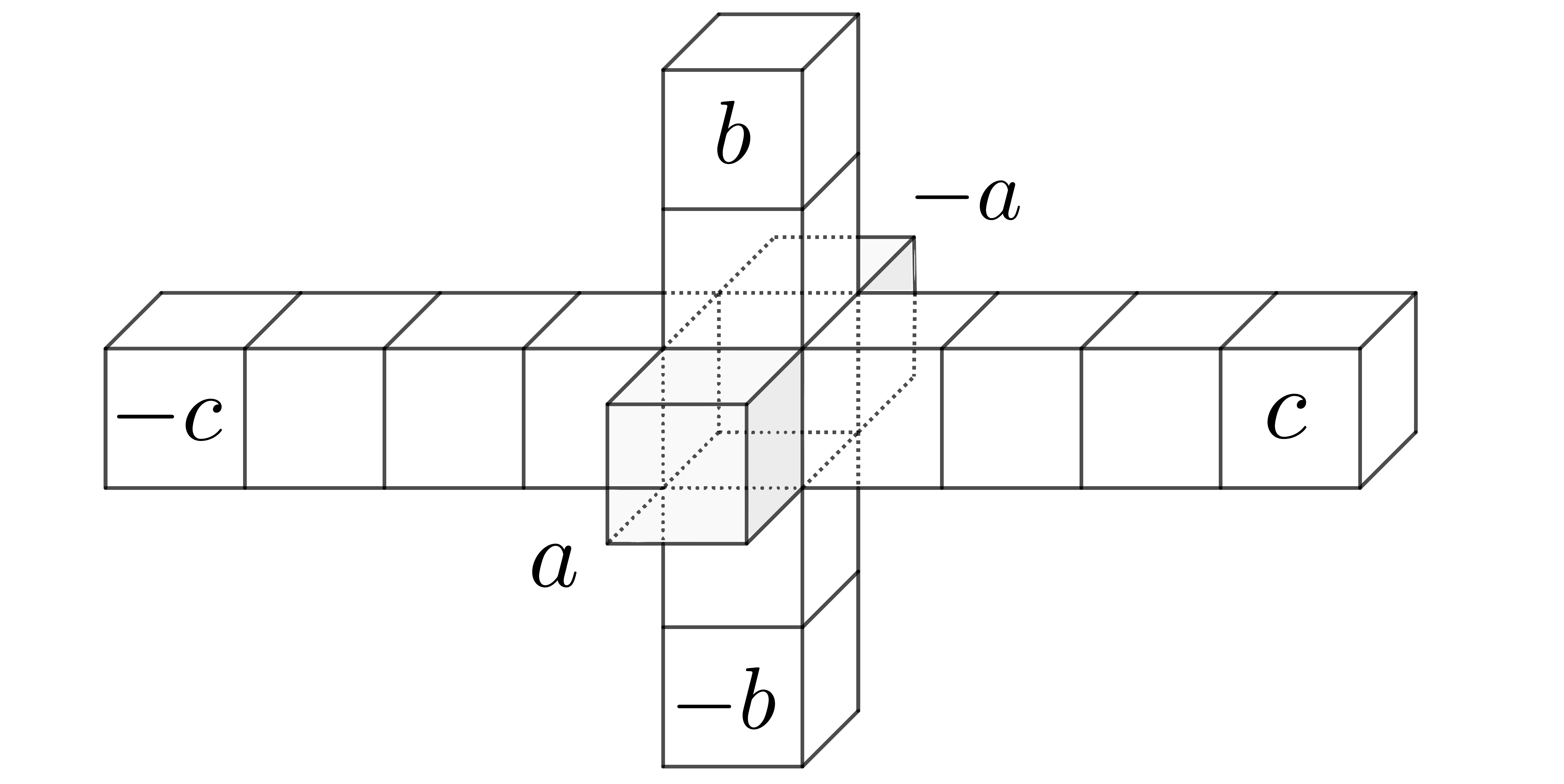}
	\caption{The neighbourhood $N_{a,b,c}$ with $a=1, b=2$ and $c=4$. The $e_1$-axis is towards the reader, the 
	$e_2$-axis is vertical, and the $e_3$-axis is horizontal.}
	\label{figanis3d}
\end{figure}

These models were studied by van Enter and Fey \cite{AA12}, and the present author \cite{DB22-1} for $r\in\{1+b+c,\dots,a+b+c\}$; 
they determined the following bounds on the critical length, as $p\to 0$,

\begin{equation}\label{vefey}
		\log \log L_c\left(\mathcal N_{r}^{a,b,c},p\right)= 
	\begin{cases}
	\Theta\left(p^{-(r-(b+c))}\right)                    & \textup{if } b=a, \\
	\Theta\left(p^{-(r-(b+c))}(\log \frac 1p)^{2}\right) & \textup{if } b>a.
	\end{cases}
	\end{equation}

We moreover determined the magnitude of the critical length up to a constant factor in the exponent in the cases $r\in\{c+1,c+2\}$, $r\le a+c$, for all triples $(a,b,c)$, except for $c=a+b-1$ when $r=c+2$ (see Section 6 in \cite{DB20}): set $s:=r-c\in\{1,2\}$, then, as $p\to 0$, 

\begin{equation}\label{myP1}
\log L_c\left(\mathcal N_{r}^{a,b,c},p\right)=  \begin{cases}
\Theta\left(p^{-s/2}\right)  & \textup{if } c=b=a, \\
\Theta\left(p^{-s/2}(\log \frac 1p)^{1/2}\right) & \textup{if } c=b>a, \\
\Theta\left(p^{-s/2}(\log \frac 1p)^{3/2}\right) & \textup{if } c\in \{b+1,\dots,a+b-s\}, \\
\Theta\left(p^{-s}\right)                                  &   \textup{if } c=a+b,   \\
\Theta\left(p^{-s}(\log \frac 1p)^{2}\right)   & \textup{if } c> a+b. 
	\end{cases}
\end{equation}

In this paper we extend the last two cases in \eqref{myP1} to all values $c<r\le b+c$, 
by determining $\log$ of the critical length up to  a  constant  factor for these values of $r$ and $c$.
The following is our main result.
\begin{theor}\label{extendBeam}
For every $r\in \{c+1,\dots,b+c\}$, as $p\to 0$, 
\begin{equation}\label{myP4}
\log L_c\left(\mathcal N_{r}^{a,b,c},p\right)=  \begin{cases}
\Theta\left(p^{-(r-c)}\right)                                  &   \textup{if } c=a+b,   \\
\Theta\left(p^{-(r-c)}(\log \frac 1p)^{2}\right)   & \textup{if } c> a+b. 
	\end{cases}
\end{equation}
\end{theor}


We also have upper bounds for the remaining values $c<a+b$ when $c<r\le a+c$, and believe that they tell us the right order of the threshold. This range of values of $r$ corresponds to two out of four possibilities for the {\it stable set} of $\mathcal N_{r}^{a,b,c}$ (see Remark \ref{Sscases}). 

\begin{prop}\label{twoout4}
Consider the sequences $\{\alpha_s\}_{s\ge 2}$ and $\{t_s\}_{s\ge 2}$  given by
\begin{equation}\label{alfas}
  \alpha_s= \frac{t+1}{t+2}\left(s-t/2\right),
  \textup{ and }
  t=t_s:=\left\lceil\frac{\sqrt{9+8s}-5}{2}\right\rceil.
\end{equation}
\begin{enumerate}
\item[\textup{(i)}] For every $r\in \{c+2,\dots,a+c\}$, if $a+b-(r-c)<c <a+b$, as $p\to 0$, 
\begin{equation}\label{myP4-3}
\log L_c\left(\mathcal N_{r}^{a,b,c},p\right)=  \begin{cases}
O\left(p^{-\alpha_{r-c}}(\log \tfrac 1p)^{2}\right)
   & \textup{if } r < a+b+\alpha_{r-c}. 
  \\
O\left(p^{-(r-(a+b))}(\log \tfrac 1p)^{2}\right)
   & \textup{if } r\ge a+b+\alpha_{r-c}. 
	\end{cases}
\end{equation}

\item[\textup{(ii)}]
For every $r\in \{c+3,\dots,a+c\}$, as $p\to 0$, \begin{equation}\label{magnitudrlea+c}
\log L_c\left(\mathcal N_{r}^{a,b,c},p\right)=  \begin{cases}
O\left(p^{-\alpha_{r-c}}\right)  & \textup{if } c=b=a, \\
O\left(p^{-\alpha_{r-c}}(\log \frac 1p)^{(t_{r-c}+1)/(t_{r-c}+2)}\right) & \textup{if } c=b>a, \\
O\left(p^{-\alpha_{r-c}}(\log \frac 1p)^{(t_{r-c}+3)/(t_{r-c}+2)}\right) & \textup{if } b<c \le a+b-(r-c).
	\end{cases}
\end{equation}
\end{enumerate}
\end{prop}

Here there are some numerical values of $t_s$ and $\alpha_s$, for $s=2,3,\dots,14$.
\begin{center}
\begin{tabular}{|c|c|c|c|c|c|c|c|c|c|c|c|c|c|} 
 \hline
 $s$ & 2 & 3 & 4 & 5 & 6& 7 & 8 & 9 & 10 & 11 & 12 & 13& 14 \\ \hline
 $t_s$ & 0 & 1 & 1 & 1 & 2 & 2 & 2 & 2 & 3 & 3 & 3 & 3 & 3 \\ \hline
 $\alpha_s$ & 1 & 5/3 & 7/3 & 3 & 15/4 & 18/4  & 21/4 & 6 & 34/5 & 38/5 & 42/5 & 46/5 & 10 \\ 
 \hline
\end{tabular}
\captionof{table}{Some values of $t_s$ and $\alpha_s$.} 
\end{center}

Finally, the range $a+c<r \le b+c$ corresponds to a third possibility for the stable set of $\mathcal N_{r}^{a,b,c}$; in these cases with $c< a+b$ (based on upper bounds), we conjecture that the critical length is given by the following.
\begin{conj}\label{targetgeneralc>b}
For $r\in\{a+c+1,\dots ,b+c\}$, as $p\to 0$,
\begin{equation}\label{magnitudrleb+c}
\log L_c\left(\mathcal N_{r}^{a,b,c},p\right)=  \begin{cases}
\Theta\left(p^{-(r-c-a+\alpha_a)}\right)  & \textup{if } c=b>a, \\
\Theta\left(p^{-(r-a-b)}
\right) 
& \textup{if } b< c < a+b, 
\end{cases}
\end{equation}
where $\alpha_1=1/2$.
\end{conj}


\subsection{The stable sets approach}
The model we study here is a special case of the following extremely general class of $d$-dimensional monotone cellular automata, which were introduced by Bollob\'as, Smith and Uzzell~\cite{BSU15}.

Let $\U=\{X_1,\dots,X_m\}$ be an arbitrary finite family of finite
subsets of $\Zd\setminus \{0\}$. We call $\U$ the {\it update family}, 
each $X\in\U$ an {\it update rule}, and the process itself $\U${\it-bootstrap percolation}.
Let $\Lambda$ be either $\Zd$ or $\Zd_L$ (the $d$-dimensional torus of sidelength $L$).
Given a set $A\subset \Lambda$ of initially {\it infected} sites, set $A_0=A$, and define for each $t\ge 0$,
\[A_{t+1}=A_t\cup\{x\in\Lambda: x+X\subset A_t \text{ for some }X\in\U\}.\]
The set of eventually infected sites is the {\it closure} of $A$, denoted by
$\genA_\U=\bigcup_{t\ge 0}A_t$, and
we say that there is {\it percolation} when $\genA_\U=\Lambda$.\\ 
For instance, our $\mathcal N_{r}^{a_1,\dots,a_d}$-model is the same as $\mathcal N_{r}^{a_1,\dots,a_d}$-bootstrap percolation,
where $\mathcal N_{r}^{a_1,\dots,a_d}$ is the family consisting of all subsets of size $r$ of the neighbourhood 
$N_{a_1,\dots,a_d}$ in (\ref{neigh3}).


 Let $S^{d-1}$ be the unit $(d-1)$-sphere and denote the discrete half space orthogonal to $u\in S^{d-1}$ as
 $\dhp^d:=\{x\in\Zd:\langle x,u\rangle <0\}$.
The {\it stable set} $\Ss=\Ss(\U)$ is the set of all $u\in S^{d-1}$
such that no rule $X\in\U$ is contained in $\dhp^d$. 
Let $\mu$ denote the Lebesgue measure on $S^{d-1}$. The following classification of families was proposed in \cite{BSU15} for $d=2$ and extended to all dimensions in \cite{BDMS15}:
A family $\U$ is

\begin{itemize}
 \item {\it subcritical} if for every hemisphere $\mathcal H \subset S^{d-1}$ we have $\mu(\mathcal H \cap\Ss)>0$.
 \item {\it critical} if there exists a hemisphere $\mathcal H \subset S^{d-1}$ such that $\mu(\mathcal H \cap\Ss)=0$, and
every open hemisphere in $S^{d-1}$ has non-empty intersection with $\Ss$;
 \item {\it supercritical} otherwise. 
 \end{itemize}



In general, we are mostly interested in critical families. It is easy to check that the family $\mathcal N_{r}^{a,b,c}$ is critical if and only if 
 $r\in\{c+1,\dots,a+b+c\}.$

\begin{rema}[The stable sets approach]\label{Sscases}
For each $i=1,2,3$, let us denote by $S_i^1:=\{(u_1,u_2,u_3)\in S^{2}: u_i=0\}$ the unit circle contained in $S^2$ that is orthogonal to the vector $e_i$.
Then, straightforward calculations lead us to
\begin{equation}\label{viaSs}
\Ss(\mathcal N_{r}^{a,b,c})=
\begin{cases}
\{\pm e_1, \pm e_2, \pm e_3\} & \textup{ for } c< r \le a+b,\\
\{\pm e_3\}\cup S^1_3 & \textup{ for } a+b< r \le a+c,\\
S^1_2 \cup S^1_3 & \textup{ for } a+c< r \le b+c,\\
S^1_1 \cup S^1_2 \cup S^1_3 & \textup{ for } b+c< r \le a+b+c.
\end{cases}
\end{equation}
The critical length $L_c\left(\mathcal N_{r}^{a,b,c},p\right)$ is determined in the case $\Ss(\mathcal N_{r}^{a,b,c})=S^1_1 \cup S^1_2 \cup S^1_3$ and by \eqref{vefey} it is doubly exponential in $p$, as $p\to 0$.
On the other hand, we have shown in \cite{DB20} that $L_c\left(\mathcal N_{r}^{a,b,c},p\right)$ is singly exponential in the first 3 cases in \eqref{viaSs}.

Given \eqref{myP1} and Theorem \ref{extendBeam}, it only remains to determine $L_c\left(\mathcal N_{r}^{a,b,c},p\right)$ 
in the first 3 cases in \eqref{viaSs} for $c<a+b$, where we believe that the magnitude is given by Proposition \ref{twoout4} and Conjecture \ref{targetgeneralc>b}. 
Note that the cases in \eqref{magnitudrleb+c} correspond to $\Ss(\mathcal N_{r}^{a,b,c})= S^1_2 \cup S^1_3$,
the cases in \eqref{myP4} and \eqref{myP4-3} correspond to $\Ss(\mathcal N_{r}^{a,b,c})= \{\pm e_3\}\cup S^1_3$, while cases in  \eqref{magnitudrlea+c} correspond to $\Ss(\mathcal N_{r}^{a,b,c})= \{\pm e_1, \pm e_2, \pm e_3\}$.

\end{rema}



\section{Proof of Theorem \ref{extendBeam}: Upper bounds}\label{SectionUpper12}
It is known that (see Proposition A.1 in \cite{DB20}) for every $r\in\{c+1,\dots,c+b\}$, as $p\to 0$,
\begin{equation}\label{generalub}
\log L_c\left(\mathcal N_{r}^{a,b,c},p\right)
=O\left(p^{-(r-c)}(\log p)^2\right).
\end{equation}

In particular, the the upper bound for the second case in \eqref{myP4} follows. Therefore, it only remains to cover the first case $c=a+b$ when $r-c\in\{3,\dots,b\}$, since the sub-cases $r-c\in\{1,2\}$ are covered by \eqref{myP1}.


\begin{defi}\label{intfilled}
A {\it rectangular block} is a set of the form $R=[l]\times[h]\times[w]\subset\Zdt$. 
A rectangular block $R$ is {\it internally filled} if $R\subset [ A\cap R]$,
and denote this event by $I^\bullet(R)$.
\end{defi}
 When $l,h,w\ge c$, for simplicity we denote the event
\[I(l,h,w) := I^\bullet([l]\times [h]\times[w]).\]


Throughout this section we will assume that
\[c=a+b.\]
As usual in bootstrap percolation, we actually prove a stronger proposition.
\begin{prop}\label{upper=}
Consider $\mathcal N_{r}^{a,b,c}$-bootstrap percolation with $r\in\{c+3,\dots,b+c\}$. 
There exists a constant $\Gamma>0$ such that, if \[L=\exp\left(\Gamma p^{-(r-c)}\right),\] then
 $\p_p\left(I^\bullet([L]^3)\right)\to 1,\ as\ p\to 0.$\end{prop} 


 \subsection{The upper bounds for $r\in \{ a+c+1, \dots, b+c\}$}
 We start with $\Ss(\mathcal N_{r}^{a,b,c})= S^1_2 \cup S^1_3$.
Let us consider the cases $a+c<r\le b+c$,
then the processes induced on the faces orthogonal to $e_3$ and $e_2$, namely $\mathcal N_{r-c}^{a,b}$ and $\mathcal N_{r-b}^{a,c}$ respectively, are supercritical, while the induced process (orthogonal to $e_1$) $\mathcal N_{r-a}^{b,c}$ is critical.
This means, that the most likely way to grow is to start with some small initially infected rectangular block and grow simultaneously along the $e_3$ and $e_2$-directions, until we reach a volume of the order of $L_c\big(\mathcal N_{r-a}^{b,c} , p\big)$, only then we can grow along the $e_1$-direction.

\begin{lemma}[Regime critical $e_1$-process]\label{regimer>a+c}
Consider $\mathcal N_{r}^{a,b,c}$-bootstrap percolation with $a+c<r\le b+c$, and
fix integers $l,h,w\ge c$. If $p$ is small enough, then
\begin{enumerate}
    \item[\textup{(i)}] $\p_p\left(I(l,h+1,w)|I(l,h,w)\right)\ge
	\left(1-e^{-\Omega(p^{r-b}w)}\right)^{a}
	\left(1-e^{-\Omega(p^{r-(a+b)}w)}\right)^{l},$
    \item[\textup{(ii)}] $\p_p\left(I(l,h,w+1)|I(l,h,w)\right)\ge
	\left(1-e^{-\Omega(p^{r-c}h)}\right)^{a}
	\left(1-e^{-\Omega(p^{r-(a+c)}h)}\right)^{l}$.
\end{enumerate}
\end{lemma}
\begin{proof}
See Lemma 2.6 in \cite{DB20}.
\end{proof}

Next, we show that if a rectangle $R\subset [L]^3$ of a well chosen size is internally filled, then it can grow and fill $[L]^3$ with high probability (in the literature  $R$ is called a {\it critical droplet}), for $L$ larger than the critical length (up to a constant factor in the exponent).
\begin{lemma}\label{lemagrow0}
Set $L=\exp(\Gamma p^{-(r-c)})$ and fix $\varepsilon>0$, $h=p^{-(r-c+\varepsilon)}$, and \[R:=[a]\times [h]\times [p^{-a}h].\] 
Conditionally on $I^\bullet(R)$,
the probability that $[L]^3$ is internally filled goes to $1$ as $p\to 0$.
\end{lemma}
\begin{proof}
Start with $R$ and apply Lemma \ref{regimer>a+c} until we reach an internally filled rectangular block $R'=[a]\times [w]^2$, with $w\approx L_c\big(\mathcal N_{r-a}^{b,c} , p\big)$, then it becomes easy to grow in all directions.
\end{proof}	

Now, we prove the upper bound for the critical length.
\begin{proof}[Proof of Proposition \ref{upper=} \textup{($r> a+c$)}]
	Set $L=\exp\big(\Gamma p^{-(r-c)}\big)$, where $\Gamma$ is a constant to be chosen. Fix a small $\varepsilon >0$, and consider the rectangle
	\[R:=\left[a\right]\times [p^{-(r-c+\varepsilon) }]\times [p^{-(r-b+\varepsilon) }] \subset[L]^3.\]
As usual, by using Lemma \ref{lemagrow0}, considering disjoint copies of $R$ in $[L]^3$ and taking $\Gamma>0$ large, it is enough to show that there exists a constant $C>0$ such that
	\begin{equation}\label{spcd0}
	\p_p(I^\bullet(R))\ge \exp(-Cp^{-(r-c)}).
	\end{equation}
	To do so, set $h=p^{-2\varepsilon}$, then for every 
	$k=1,\dots,n:=p^{-(r-c)+\varepsilon}$ set
	\[h_k=hk,\ \ w_k=p^{-a}h_k,\ \ R_k=[a]\times [h_k]\times[w_k], \text{ and } \ R_k'=[a]\times [h_k]\times[w_{k+1}].\]
	Note that $R_n=R$,
	$h_{k+1}=h_k+h$ and $w_{k+1}=w_k+p^{-a}h$,
	so by Lemma \ref{regimer>a+c},

	\begin{align*}
	\p_p(I^\bullet(R_n)) & \ge \p_p(R_1\subset A)\prod_{k=1}^{n-1}
	\p_p(I^\bullet(R_k')|I^\bullet(R_k))
	\p_p(I^\bullet(R_{k+1})|I^\bullet(R_{k}'))\\
	& \ge p^{|R_1|}
\prod_{k=1}^{n}	\left[ 
	\left(1-e^{-\Omega(p^{r-c}h_{k})}\right)^{a}
	\left(1-e^{-\Omega(p^{r-c-a}h_{k})}\right)^{a}
	\right]^{p^{-a}h}\\
	& \ \ \ \times 
\prod_{k=1}^{n}	\left[ 
	\left(1-e^{-\Omega(p^{r-b}w_{k+1})}\right)^{a}
	\left(1-e^{-\Omega(p^{r-b-a}w_{k+1})}\right)^{a}
	\right]^{h}\\
	& \ge p^{|R_1|}
	p^{O(n)}
\prod_{k=1}^{n}	\left[ 
1-e^{-\Omega(p^{r-2a-b}hk)}
	\right]^{2p^{-a}h}.
\end{align*}
Finally, note that $|R_1|= ap^{-a-4\varepsilon}\ll O(n)$, since $r>a+c$, thus
\begin{align*}
\p_p(I^\bullet(R))    	& \ge  
	p^{O(n)}
\exp\left(-\Omega\left(p^{-a}p^{-(r-2a-b)}\int_{0}^{\infty}[-\log(1-e^{-x})]\,dx\right)\right)\\
& \ge \exp\big(-Cp^{-(r-(a+b))}\big),
	\end{align*}
	for some constant $C>0$, as we claimed.
\end{proof}

Before we prove Proposition \ref{upper=} when $r\le a+c$, we will have a quick discussion about supercritical two-dimensional families.

\subsection{The supercritical families $\mathcal N_s^{s,s}$}
In this section we consider two-dimensional supercritical $\mathcal N_s^{s,s}$-bootstrap percolation and assume that $s\ge 3$.
\begin{defi}\label{seqpattern}
An $s$-{\it pattern} is a union of $t+1$ sets of vertices:
\[S_0\cup S_{1}\cup \cdots\cup S_t,\]
where $t=t_s<s$ is the biggest integer satisfying

\begin{equation}\label{valueoft}
    \frac{|S_0|+ |S_{1}|+ \cdots +|S_t|
    }{t+2} < s-t,
\end{equation}
and for each $i=0,1,\dots,t$, $S_i\subset \{i+1\}\times\mathbb Z$ is a copy of $\{1\}\times[s-i]$ (so that $|S_i|=s-i$) in the following restricted sense (recall that $e_2=(0,1)$)
\[S_i=m(s-i)e_2 + \{i+1\}\times[s-i], \textup{ for some integer }m\ge 0.\]

\begin{rema}\label{reasonrestricted}
The restrictions above are made to guarantee that when two $s$-patterns intersect in some column $i$, this necessarily implies that they coincide in column $i$. This fact and independence imply that the probability of existing a set $S_i$ inside $\{i+1\}\times[k]$ is at least
\begin{equation}\label{costofSi}
  1-\exp\left(-\Omega\left(p^{|S_i|}k\right)\right).  
\end{equation}
\end{rema}

The next step is to provide a lower bound for the probability of the event $I^{\bullet}([l]\times [k])$. This is the main lemma for $s\ge 3$.

\begin{lemma}[Supercritical induced process]\label{MainLemmafor>2}
Fix $m\in [s]\cup \{0\}$. Under $\mathcal N_s^{s,s}$-bootstrap percolation, there exists $\delta>0$ such that, if $k=\Omega(p^{-m})$ then
\begin{equation}\label{intfillingprob'}
\p_p(I^{\bullet}([l]\times [k])) \ge 
1-\left(  1- \delta\prod_{i=m+1}^s
 \left[  1-\exp\left(-\Omega\left(kp^i\right)\right)  \right]\right)^{l/s}.
\end{equation}
If moreover, $k\le (2/3) p^{-(m+1)}$, then 
\begin{equation}\label{intfillingprob''}
\p_p(I^{\bullet}([l]\times [k])) \ge 
  1-\exp\left(-\Omega\left(k^{s-m}lp^{\sum_{i=m+1}^si}\right)\right).
\end{equation}
\end{lemma}

\begin{proof}
Partition the rectangle $R=[l]\times [k]$ 
into $l/s$ copies of $R'=[s]\times [k]$, and note that $R$ is internally filled if we can find $s$ restricted (in the sense of Definition \ref{seqpattern}) sets
\[S_0 \cup S_{1}\cup  \cdots \cup S_{s-1}\]
 in the rectangle $R'$ (or any of its disjoint copies), so by Remark \ref{reasonrestricted} it follows that
\begin{equation*}
  \p_p(I^\bullet (R))\ge
  1-\left(  1- \prod_{i=1}^s
 \left[  1-\exp\left(-\Omega\left(kp^i\right)\right)  \right]\right)^{l/s}.
\end{equation*}
Now, if $k=\Omega(p^{-m})$ then 
$
\prod_{i=1}^m
 \left[  1-\exp\left(-\Omega\left(kp^i\right)\right)  \right]
 \ge \delta,
$
this proves (\ref{intfillingprob'}).

Finally, if $k\le (2/3) p^{-(m+1)}$, then for every $i\ge m+1$ we have $kp^i\le 2/3$, hence
\begin{equation*}
 \prod_{i=m+1}^s
 \left[  1-\exp\left(-\Omega\left(kp^i\right)\right)  \right]
 \ge \Omega\left(\prod_{i=m+1}^s kp^i\right)
 =\Omega\left(k^{s-m}p^{\sum_{i=m+1}^si}\right),
\end{equation*}
and (\ref{intfillingprob''}) follows by applying $1-q\ge e^{-2q}$ for $q$ small.
\end{proof}

\subsection{The upper bounds for $r\in \{c+3, \dots , a+c\}$}
Now, let us consider our $\mathcal N_r^{a,b,c}$-bootstrap percolation with $c+2<r\le a+c$, set
\[s:= r-c\ge 3.\]
This is a consequence of Lemma \ref{MainLemmafor>2}.
\begin{coro}\label{corosp}
Consider $\mathcal N_r^{a,b,c}$-bootstrap percolation and fix $m\in [s]\cup \{0\}$. If $p$ is small enough and $p^{-m}< h \le  p^{-(m+1)}$, then 
\begin{equation}\label{corospatt}
\p_p\left(I(l,h,w+1)|I(l,h,w)\right) \ge 
  1-\exp\left(-\Omega\left(h^{s-m}lp^{\sum_{i=m+1}^si}\right)\right).
\end{equation}
\end{coro}
\begin{proof}
The induced $\mathcal N_{r-c}^{a,b}$-process along the $e_3$-direction (on the $[l]\times [h]$ face) is dominated by the $\mathcal N_{s}^{s,s}$-process, since $s=r-c$ and $s\le a\le b$, so Lemma \ref{MainLemmafor>2} applies.
\end{proof}

This corollary tells us the cost of growing along the (easiest) $e_3$-direction, and we are also interested in computing the cost of growing along the $e_1$ and $e_2$ (harder) directions.\\
Note that $r>c\ge a+b$ in Theorem \ref{extendBeam}, so we just need to cover the cases $a+b<r\le a+c$,
where, all induced processes $\mathcal N_{r-c}^{a,b}$, $\mathcal N_{r-b}^{a,c}$ and $\mathcal N_{r-a}^{b,c}$ are supercritical.

\begin{lemma}[Regime supercritical $e_1$-process]\label{regimer>a+b}
Consider $\mathcal N_{r}^{a,b,c}$-bootstrap percolation with $a+b<r\le a+c$, and
fix integers $l,h,w\ge c$. If $p$ is small enough, then
\begin{enumerate}
    \item[\textup{(i)}] $\p_p\left(I(l,h+1,w)|I(l,h,w)\right)\ge
	\left(1-e^{-\Omega(p^{r-b}w)}\right)^{a}
	\left(1-e^{-\Omega(p^{r-(a+b)}w)}\right)^{l},$
    \item[\textup{(ii)}] $\p_p\left(I(l+1,h,w)|I(l,h,w)\right)\ge
	\left(1-e^{-\Omega(p^{r-a}w)}\right)^{b}
	\left(1-e^{-\Omega(p^{r-(a+b)}w)}\right)^{h}$.
\end{enumerate}
\end{lemma}
\begin{proof}
See Lemma 2.6 in \cite{DB20}.
\end{proof}

Next, we show the candidate to be our critical droplet.

\begin{lemma}\label{lemagrow}
Set $L=\exp(\Gamma p^{-s})$ and fix $\varepsilon>0$, $l=p^{-(1+\varepsilon)}$, and \[R:=[l]\times [l^{s-1}]\times [l^s].\] 
Conditionally on $I^\bullet(R)$,
the probability that $[L]^3$ is internally filled goes to $1$ as $p\to 0$.
\end{lemma}
\begin{proof}
Follows from Lemma \ref{regimer>a+b} and Corollary \ref{corosp} with $m\ge s-1$.
\end{proof}	

Now, we prove the upper bound for the critical length.
\begin{proof}[Proof of Proposition \ref{upper=} \textup{($r \le a+c$)}]
	Set $L=\exp(\Gamma p^{-s})$, where $\Gamma$ is a constant to be chosen. Fix a small $\varepsilon >0$, and consider the rectangle
	\[R:=\left[p^{-(1+\varepsilon)}\right]\times [p^{-(1+\varepsilon) (s-1)}]\times [p^{-(1+\varepsilon) s}] \subset[L]^3.\]
As usual, by using Lemma \ref{lemagrow}, considering disjoint copies of $R$ in $[L]^3$ and taking $\Gamma>0$ large, it is enough to show that there exists a constant $C>0$ such that
	\begin{equation}\label{spcd}
	\p_p(I^\bullet(R))\ge \exp(-Cp^{-s}).
	\end{equation}
	To do so, set $l=p^{-\varepsilon}$, then for every 
	$k=1,\dots,n:=p^{-1}$ set
	\[l_k=kl,\ h_k=l_k^{s-1},\ w_k=l_k^s,\ \ R_k=[l_k]\times [h_k]\times[w_k],\]
	\[R_k'=[l_k]\times [h_k]\times[w_{k+1}] \text{ and}\ R_k''=[l_k]\times [h_{k+1}]\times[w_{k+1}].\]
	Note that $R_n=R$,
	$l_{k+1}=l_k+l$, $h_{k+1}=h_k+O(k^{s-2}l^{s-1})$ and $w_{k+1}=w_k+O(k^{s-1}l^{s})$,
	so by Lemma \ref{regimer>a+b} and Corollary \ref{corosp} with $m\le s-1$,
	we have
	
	\begin{align*}
	\p_p(I^\bullet(R_n)) & \ge \p_p(R_1\subset A)\prod_{k=1}^{n-1}
	\p_p(I^\bullet(R_k')|I^\bullet(R_k))
	\p_p(I^\bullet(R_k'')|I^\bullet(R_k'))
	\p_p(I^\bullet(R_{k+1})|I^\bullet(R_{k}''))\\
	& \ge p^{|R_1|}\prod_{m=0}^{s-1}\prod_{k=l^{-1}p^{-\frac{m}{s-1}}}^{l^{-1}p^{-\frac{m+1}{s-1}}} \left( 1-\exp\left[-\Omega\left(h_k^{s-m}l_kp^{\sum_{i=m+1}^si}\right)\right]\right)^{O(k^{s-1}l^s)}\\
	& \ \ \ \times 
\prod_{k=1}^{n}	\left[ 
	\left(1-e^{-\Omega(p^{r-b}w_{k+1})}\right)^{r}
	\left(1-e^{-\Omega(p^{s}w_{k+1})}\right)^{l_k}
	\right]^{O(k^{s-2}l^{s-1})}\\
	& \ \ \ \times 
\prod_{k=1}^{n}	\left[ 
	\left(1-e^{-\Omega(p^{r-a}w_{k+1})}\right)^{r}
	\left(1-e^{-\Omega(p^{s}w_{k+1})}\right)^{h_{k+1}}
	\right]^{l}\\
	& \ge p^{|R_1|}\prod_{m=0}^{s-1}\prod_{k\ge 0} \left( 1-\exp\left[-\Omega\left((kl)^{(s-1)(s-m)+1}p^{\sum_{i=m+1}^si}\right)\right]\right)^{O(k^{s-1}l^s)}\\
	& \ \ \ \times 
	p^{O((nl)^{s-1})} \times p^{O(nl)}
\prod_{k=1}^{n}	\left[ 
1-e^{-\Omega(p^{s}(kl)^s)}
	\right]^{O(k^{s-1}l^{s})}
\end{align*}
\begin{align*}
    	& \ge \prod_{m=0}^{s-1}\exp\left(-\Omega\left(p^{-{s\frac{\sum_{i=m+1}^s i }{(s-1)(s-m)+1}}}\int_{0}^{\infty}z^{s-1}[-\log(1-e^{-z^{(s-1)(s-m)+1}})]\,dz\right)\right)\\
	& \ \ \ \times 
	p^{O(p^{-(1+\varepsilon)(s-1)})}
\exp\left(-\Omega\left(p^{-s}\int_{0}^{\infty}z^{s-1}[-\log(1-e^{-z^s})]\,dz\right)\right)\\
& \ge \exp(-Cp^{-s}),
	\end{align*}
	for some constant $C>0$, since $m\le s-1$ implies $(s-m-1)(s-m-2)\ge 0$, so that $\sum_{i=m+1}^s i \le (s-1)(s-m)+1$, and we are finished.
\end{proof}


\section{Lower bounds}\label{SectionLower}  
To prove the lower bounds, we will use a technique introduced in \cite{DB20} called the {\it beams process}. 
\subsection{Subcritical two-dimensional families}
In this section, we recall an exponential decay property that holds for subcritical families $\UU$ with $\Ss(\UU)=S^1$. Consider
$\UU$-bootstrap percolation in $\Zdd$ with $\UU$ subcritical.

\begin{defi}
We define the {\it component} (or {\it cluster})
of $0\in\Zdd$ as the connected component containing $0$ in the graph induced by 
$\langle A\rangle_\UU$, and we denote it by $\mathcal K=\mathcal K(\UU,A)$.
If $0\notin \genA_\UU$, then we set $\mathcal K=\emptyset$.
\end{defi}
The following result was proved in \cite{DB20}.

\begin{theor}[Exponential decay for the cluster size]\label{expdecay}
Consider subcritical $\UU$-bootstrap percolation on $\Zdd$ with $\Ss(\UU)=S^1$.
If $\varepsilon >0$ is small and $C=C(\varepsilon):=\log(\frac{1}{\varepsilon})$, then
	\begin{equation}
	\p_\varepsilon(|\mathcal K|\ge n)\le 
 \varepsilon^{\Omega(n)}=e^{-\Omega(Cn)},
	\end{equation}
	for every $n\in\mathbb N$.
\end{theor}
\begin{proof}
See Thorem 4.11 in \cite{DB20}.
\end{proof}
Observe that $\Ss(\mathcal N_{m}^{a,b})=S^1$   if and only if $m\ge a+b+1$,
in particular, our exponential decay result (Theorem \ref{expdecay}) holds for these families.

\subsection{The beams process}
From now on we set
\begin{equation}
m:=a+b+1. 
\end{equation}

\begin{defi}\label{beam}
A {\it beam} is a finite subset of $\Zdt$ of the form $H\times[w]$, where $H\subset\Zdd$ is 
 connected and $\langle H\rangle_{\mathcal N_{m}^{a,b}}=H$.
\end{defi}

In order to introduce the beams process, we need more definitions.

\begin{defi}\label{genbeam}
Given a beam $H\times[w]$ and sets $S_1,S_2\subset  H\times[w]$, we say that $H\times[w]$ is {\it generated by} $S_1\cup S_2$ if the sets $H_1,H_2\subset H$ given by
 \[H_i:=\{x\in R: (\{x\}\times[w])\cap S_i\ne\emptyset\},\ \ \ i=1,2,\]    
are connected and there exists a path $P\subset H$  with minimal size ($P$ could be $\emptyset$), connecting $H_1$ to $H_2$, such that $H=\langle H_1\cup H_2\cup P\rangle_{\mathcal N_{m}^{a,b}}$.
Moreover, we denote \[B(S_1\cup S_2):=H\times [w].\]
\end{defi}
In this definition $\langle S_1\cup S_2\rangle\subset H\times[w]$ for each $r\ge m$, and generated beams could depend on the choice of the path $P$. However, such minimal paths are not relevant for our purposes.

\begin{exam}
In Figure \ref{figBEAM} we show (disconnected) sets $S_1, S_2$ on the picture to the left, and the beam $B(S_1\cup S_2)$ with respect to the subcritical family $\mathcal N_{4}^{1,2}$ to the right. $S_1$ consists of the left-most isolated vertex union the copy of $\{e_3,2e_3,3e_3\}$ (three consecutive vertices) on the top right-most side, while $S_2$ consists of all remaining vertices.\\
Following the notation in Definition \ref{genbeam}, note that $H_1$ and $H_2$ are connected, and we can take $P=\emptyset$ since $\langle H_1\cup H_2\rangle_{\mathcal N_{4}^{1,2}}=H$ is already connected.
\begin{figure}[ht]
	\includegraphics[width=5cm, height=3.5cm]{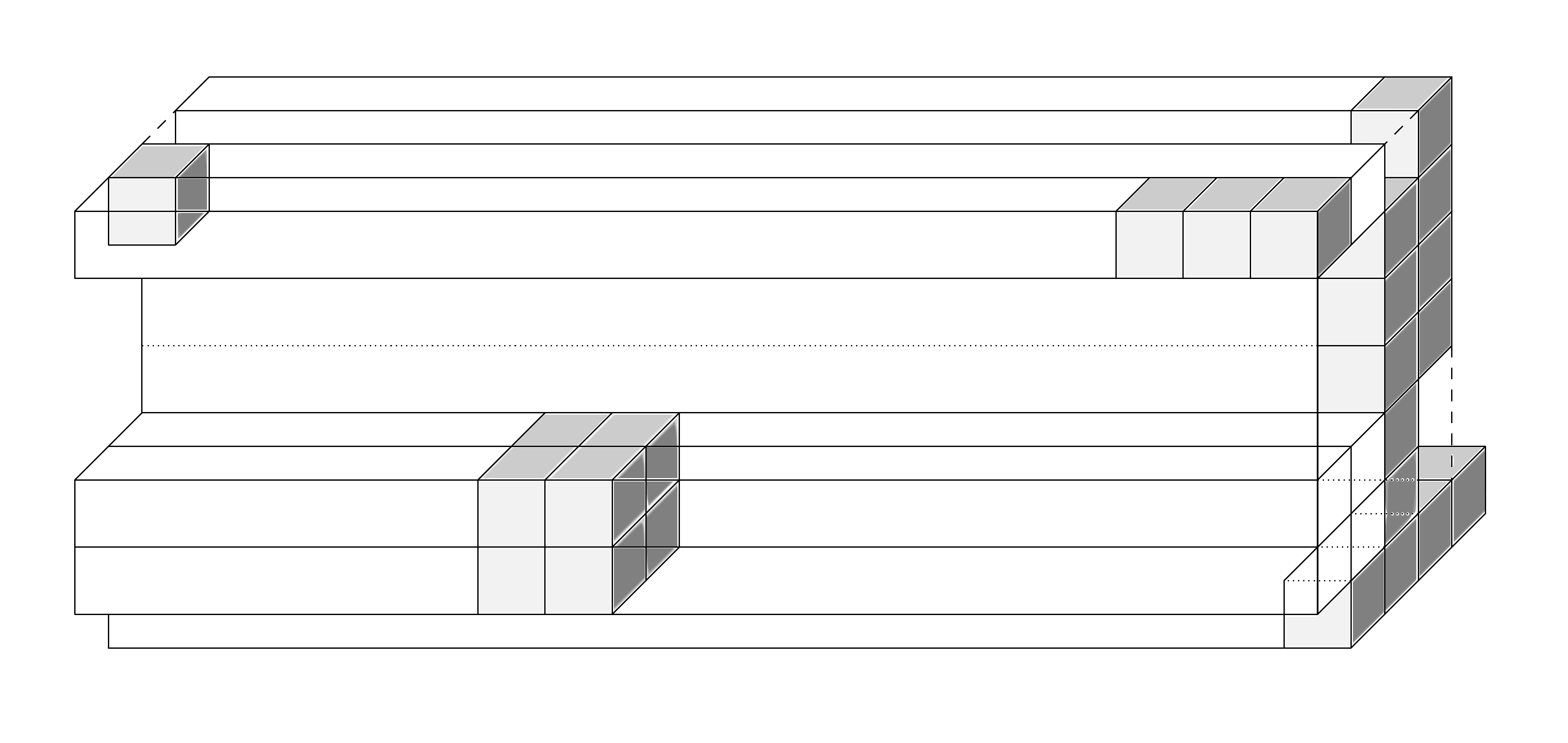} 
\hspace{2.5cm}
	\includegraphics[width=5cm, height=3.5cm]{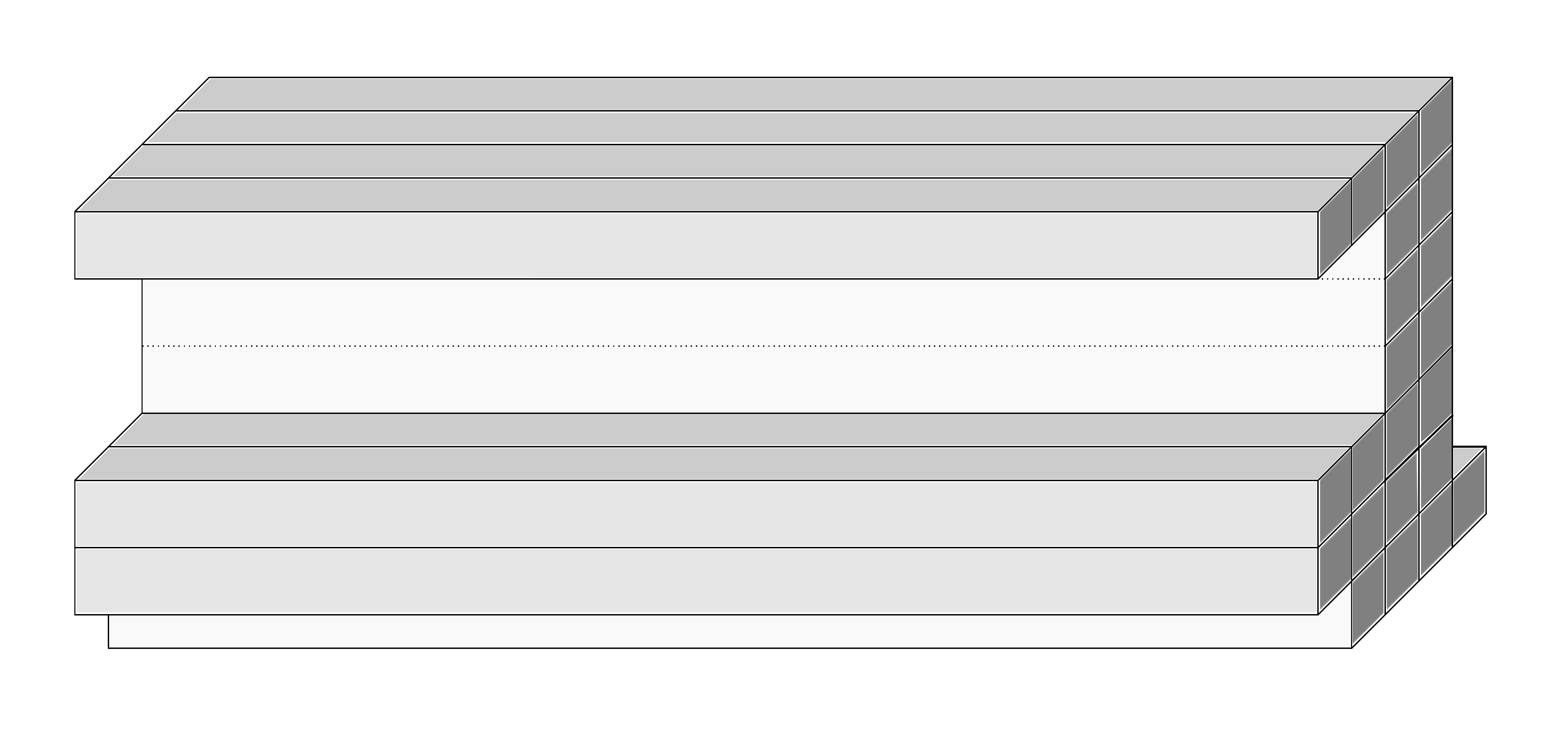}
	\caption{A generated beam w.r.t. the subcritical family $\mathcal N_{4}^{1,2}$.}
	\label{figBEAM}
\end{figure}
\end{exam}

Next, let us consider the following coarser process.
\begin{defi}[Coarse bootstrap percolation]\label{coarse}
Partition $[L]^2$ as $L^2/(b+1)^2$ copies of
 $\boxplus:=[b+1]^2$ in the obvious way, and think of $\boxplus$ as a single vertex in the new scaled grid $[L/(b+1)]^2$. 
 Given a two-dimensional family $\UU$, suppose we have some fully infected copies of $\boxplus\in [L/(b+1)]^2$ and denote this initially infected set by $A$, then we define
 {\it coarse} $\UU$-bootstrap percolation to be the result of applying $\UU$-bootstrap percolation
 to the new rescaled vertices. We denote the closure of this process by $\genA_{b}$.
\end{defi}

To avoid trivialities, we assume that $b+1$ divides $L$. 
\begin{defi}\label{coarbeam}
A {\it coarse beam} is a finite set 
of the form $H\times[w]$,
 where $H\subset\Zdd$ is connected and $[ H]_{b}=H$
 under coarse $\mathcal N_{m}^{a,b}$-bootstrap percolation.
\end{defi}
\begin{nota}\label{constru2}
Given sets $S_1,S_2\subset [L]^2\times[L]$, we partition $[L]^2$ as in Definition \ref{coarse} and denote by $B_b(S_1\cup S_2)$ the coarse beam generated by $S_1\cup S_2$ which is constructed in the (coarse) analogous way, as we did in Definition \ref{genbeam}, using coarse paths when needed. Note that every coarse beam is a beam.
\end{nota}
\begin{defi}\label{strconn}
Let $G^3=(V,E)$ be the graph with vertex set $[L]^3$ and edge set given by
$E=\{uv: \|u-v\|_\infty \le 2c\}$.   
 We say that a set $S\subset [L]^3$ is {\it strongly connected}
 if it is connected in the graph $G^3$.
\end{defi}

We use the beams process to show an Aizenman-Lebowitz-type lemma which says that when $[L]^3$
is internally filled, then it contains {\it covered} beams of all relevant intermediate sizes (see Lemma \ref{ALlema2} below).

\begin{defi}[The coarse beams process]
Let $A=\{x_1,\dots,x_{|A|}\}\subset [L]^3$ and fix $r\ge c+1$.
 Set $\mathcal B:=\{S_1,\dots,S_{|A|}\}$, where $S_i=\{x_i\}$ for each $i=1,\dots,|A|$, and
repeat until STOP:
 \begin{enumerate}
  \item If there exist distinct beams $S_{1},S_{2}\in\mathcal B$ such that
  \[S_{1}\cup S_{2}\]
  is strongly connected, and $\langle S_{1}\cup S_{2} \rangle\ne 
  S_{1}\cup  S_{2}$,
  then  remove them from $\mathcal B$,
  and replace by a coarse beam $B_b(S_{1}\cup S_{2})$. 
  \item If there do not exist such a family of sets in $\mathcal B$, then STOP.
 \end{enumerate}
We call any beam $S=B_b(S_{1}\cup  S_{2})$ added to the collection $\mathcal B$ a {\it covered beam}, and denote the event that $S$ is covered by
$I_b^{\text{\ding{86}}}(S).$
 \end{defi}

Consider $\mathcal N_{r}^{a,b,c}$-bootstrap percolation with $r\ge c+1$, and let $\kappa, \lambda$ be large constants depending on $b,c$ and $r$. The following is a beams version of the Aizenman-Lebowitz lemma.

\begin{lemma}\label{ALlema2}
 If $[L]^3$ is internally filled then for every $h,k=\kappa,\dots,L$, there exists
 a covered (coarse) beam
 $H\times[w]\subset [L]^3$ satisfying $w\le \lambda k$, $|H|\le \lambda h$, and
 either $w\ge k$ or $|H|\ge h$.
\end{lemma}

\begin{proof}
See Lemma 5.13 in \cite{DB20}.
\end{proof}

Let $\mathcal B_{h,k}=\mathcal B_{h,k}(L)$ be the collection of all connected sets of the form $H\times[w]$ contained in $[L]^3$ satisfying $|H|\le h$ and $w\le k$.
The proof of Lemma 5.2 in \cite{DB20} implies the following upper bound for $\mathcal B_{h,k}$. 
\begin{lemma}[Upper bound on the number of beams]\label{nofbeams}
	For all $h,k\le L$,
	\[|\mathcal B_{h,k}|\le L^4(3e)^{h}.\]
\end{lemma}

\subsection{The lower bounds}
The proofs in this section work for all values $r>a+b$. However, by \eqref{vefey}, they are useless when $r>b+c$.

\subsubsection{The case $c=a+b$}
In this section we prove the following.
\begin{prop}\label{lower3}
Under $\mathcal N_{r}^{a,b,c}$-bootstrap percolation with $r\in\{c+2,\dots,b+c\}$ and $c=a+b$, there is a constant $\gamma=\gamma(a,b)>0$
such that, if 
\[L<\exp(\gamma p^{-(r-c)}),\]
then 
$\p_p[I^{\bullet}([L]^3)]\to 0,\ as\ p\to 0.$
\end{prop}

\begin{proof}
Set $s=r-c$ and take $L<\exp(\gamma p^{-s})$, where $\gamma>0$ is some small constant. Let us show that
$\p_p(I^{\bullet}([L]^3))$ goes to $0$, as $p\to 0$. Fix $\varepsilon>0$.

If $[L]^3$ is internally filled, by Lemma \ref{ALlema2} with 
\[h=k=\varepsilon/\lambda p^s,\]
there exists a covered (coarse) beam $S=H\times[w]\subset [L]^3$
satisfying $w,|H|\le\varepsilon/p^s$, and moreover, either $w\ge k$ or $|H|\ge h$, hence, by union bound, $\p_p[I^{\bullet}([L]^3)]$ is at most
\begin{align*}
\sum_{S\in\mathcal B_{\lambda h,\lambda k}}(\p_p[I^{\text{\ding{86}}}(S)\cap
\{w\ge k\}]+\p_p[I^{\text{\ding{86}}}(S)\cap \{|H|\ge  h\}]).
\end{align*}

To bound the first term, we use the fact that $H\times[w]$ is covered; this implies that every copy (inside $S$) of the slab $H\times [rs]$ must contain $s$ vertices in $A$ within constant distance.
Therefore, by considering the $w/rs$ disjoint slabs  (and applying the FKG inequality to the $O(|H|)$ distinct subsets $Z$ of $s$ vertices within constant distance, $Z\subset H\times [rs]$), if $\varepsilon$ is small, then there exists some $c_1=c_1(\varepsilon,r)>0$ such that

\begin{align*}
\p_p[I^{\text{\ding{86}}}(H\times[w])\cap
\{w\ge k\}]  \le
\left(1-e^{-\Omega(p^s|H|)}\right)^{k/rs} 
 =\left(1-e^{-\Omega(\varepsilon)}\right)^{k/rs}
 \le e^{-c_1/p^s}.
\end{align*}


To bound the second term we use the fact that
if $[L]^3$ is internally filled, then every copy of
$\boxplus\times [L]$
(as in Definition \ref{coarse}) should contain at least \[t:=r-(a+b)\]
vertices $v_0,v_1,\dots,v_{t-1}$ with $\|v_i-v_l\|=O(1)$ for all $i,l< t$ such that, some $l<t$ satisfies $v_0,v_1,\dots, v_l\in A$ and $\forall i>l$, $v_i$  got infected by using $r-i$
infected neighbours in $v_i+N_{a,b}\times \{0\}$, where
$N_{a,b}$ is given by \eqref{neigh3} (so, $v_i\in [\{v_0,\dots, v_{i-1} \} \cup ([L]^3\setminus (\boxplus\times [L]))]$, otherwise, there is no way to fully infect such a copy).\\
Moreover, by our choice of $t$, $\mathcal N_{r-i}^{a,b}$ is subcritical for all $i$.
Therefore, by monotonicity we can couple the process on $[L]^2\times[w]$ (wlog we are assuming that $S\subset [L]^2\times[w]$) having initial infected set
$A$, with coarse $\mathcal N_{m}^{a,b}$-bootstrap
percolation ($m=a+b+1$) on $[L/(b+1)]^2\times\{0\}\subset\Zdd$ and initial infected set
\[A':=\left\{\boxplus\in[L]^2: |A\cap (\boxplus\times [w])|_{O(1)}\ge t
\right\},\]
where the subindex $O(1)$ in the cardinality symbol means that the vertices participating in the intersection are within constant distance. \\
Now, by applying Markov's inequality and the fact that $c=a+b$ implies $t=s$,
\[\p_p(|A\cap(\boxplus\times [w])|_{O(1)}\ge t) = 
O(wp^t)\le O(\varepsilon).\]
In particular, under (coarse) $\mathcal N_{m}^{a,b}$-bootstrap percolation with initial infected set $\varepsilon$-random,  there should exist a connected component of size at least
$|H|\ge h$ inside $[L]^2$. On the other hand, there are at most
$L^2$ possible ways to locate the origin in $H$, so if $\mathcal K$ denotes the (coarse) cluster of 0, Theorem \ref{expdecay} implies

\begin{align*}
\p_p[I_b^{\text{\ding{86}}}(S)\cap \{|H|\ge h\}]
& \le \sum_{\boxplus\subset [L]^2}\p_\varepsilon(\{|\mathcal K|\ge h\}\cap \{\boxplus=0\}) 
 \le L^{2}\p_\varepsilon(|\mathcal K|\ge h)\\
& \le  e^{2\gamma/p^s}e^{-\Omega(C\varepsilon/ p^s)} 
 = e^{-\Omega(p^{-s})},  
\end{align*}
where $C=-\log\varepsilon$ and, we choose small $\varepsilon>0$ such that $C\varepsilon>0$ and $\gamma\ll C\varepsilon$ at first.
By Lemma \ref{nofbeams} we conclude that

\begin{align*}
\p_p[I^{\bullet}([L]^3)] & \le \sum_{S\in\mathcal B_{\frac{\varepsilon}{p^s},\frac{\varepsilon}{p^s}}}\{ e^{-c_1/p^s}+e^{-\Omega(p^{-s})}  \} 
 \le L^4(3e)^{\varepsilon p^{-s}}e^{-\Omega(p^{-s})}
\to 0,
\end{align*}
for $\gamma>0$ small enough, since $L<e^{\gamma p^{-s}}$.
\end{proof}

\subsubsection{The case $c>a+b$}
  In this section we prove the lower bound corresponding to our last case $c>a+b$. 
We will show the following.
\begin{prop}\label{lower4}
Under $\mathcal N_{r}^{a,b,c}$-bootstrap percolation with $r\in\{c+2,\dots,b+c\}$ and $c>a+b$, there is a constant $\gamma=\gamma(c)>0$
such that, if 
\[L<\exp(\gamma p^{-(r-c)}(\log p)^2),\]
then 
$\p_p[I^{\bullet}([L]^3)]\to 0,\ as\ p\to 0.$

\end{prop}
\begin{proof}
Set $s=r-c$ and take $L<\exp(\gamma p^{-s}(\log p)^2)$, where $\gamma>0$ is some small constant. Let us show that
$\p_p(I^{\bullet}([L]^3))$ goes to $0$, as $p\to 0$. Fix $\delta >0$ and set
\[h=(\delta p^{-s}\log \tfrac{1}{p})/\lambda, \ \ \ \ \ k=p^{-s-1/2}/\lambda,\]

If $[L]^3$ is internally filled, by Lemma \ref{ALlema2} with 
there exists a covered (coarse) beam $S=H\times[w]\subset [L]^3$
satisfying $w\le \lambda k,\ |H|\le\lambda  h$, and moreover, either $w\ge k$ or $|H|\ge h$, hence, 
\begin{align*}
\p_p[I^{\bullet}([L]^3)]\le 
\sum_{S\in\mathcal B_{\lambda h,\lambda k}}(\p_p[I^{\text{\ding{86}}}(S)\cap
\{w\ge k\}]+\p_p[I^{\text{\ding{86}}}(S)\cap \{|H|\ge  h\}]).
\end{align*}

To bound the first term, we use the fact that $H\times[w]$ is covered; this implies that every copy (inside $S$) of the slab $H\times [rs]$ must contain $s$ vertices in $A$ within constant distance.
Therefore,  if $\delta$ is small enough, then 
\begin{align*}
\p_p[I^{\text{\ding{86}}}(H\times[w])\cap
\{w\ge k\}]  \le
\left(1-e^{-\Omega(p^s|H|)}\right)^{k/rs} 
 \le \left(1-p^{O(\delta)}\right)^{k/rs}
 \le e^{-\Omega(p^{-s}(\log p)^2)}.
\end{align*}


To bound the second term we use the fact that
if $[L]^3$ is internally filled, then every copy of
$\boxplus\times [L]$
(as in Definition \ref{coarse}) should contain at least \[t:=r-(a+b)\]
vertices $v_0,v_1,\dots,v_{t-1}$ with $\|v_i-v_l\|=O(1)$ for all $i,l< t$ such that, some $l<t$ satisfies $v_0,v_1,\dots, v_l\in A$ and $\forall i>l$, $v_i$  got infected by using $r-i$
infected neighbours in $v_i+N_{a,b}\times \{0\}$.
Moreover, $\mathcal N_{r-i}^{a,b}$ is subcritical for all $i$.
Therefore, by monotonicity we can couple the process on $[L]^2\times[w]$ having initial infected set
$A$, with coarse $\mathcal N_{m}^{a,b}$-bootstrap
percolation on $[L/(b+1)]^2\times\{0\}\subset\Zdd$ and initial infected set
\[A':=\left\{\boxplus\in[L]^2: |A\cap (\boxplus\times [w])|_{O(1)}\ge t
\right\},\]
where the subindex $O(1)$ in the cardinality symbol means that the vertices participating in the intersection are within constant distance. \\
Now, by applying Markov's inequality and the fact that $c>a+b$ implies $t\ge s+1$,
\[\p_p(|A\cap(\boxplus\times [w])|_{O(1)}\ge t) = 
O(wp^t)= O(p^{-s-1/2+t})\le O(p^{1/2}).\]
In particular, under (coarse) $\mathcal N_{m}^{a,b}$-bootstrap percolation with initial infected set $p^{1/2}$-random,  there should exist a connected component of size at least
$|H|\ge h$ inside $[L]^2$. 
So, if $\mathcal K$ denotes the (coarse) cluster of 0, Theorem \ref{expdecay} implies for $\gamma \ll \delta$,

\begin{align*}
\p_p[I_b^{\text{\ding{86}}}(S)\cap \{|H|\ge h\}]
 \le L^{2}\p_{p^{1/2}}(|\mathcal K|\ge h)
 \le  e^{2\gamma p^{-s}(\log p)^2}e^{-\Omega(C  h)} 
 = e^{-\Omega(p^{-s}(\log\frac 1p)^2)},  
\end{align*}
where $C=-\log(p^{1/2})=\Theta(\log\frac 1p)$.
By Lemma \ref{nofbeams} we conclude that

\begin{align*}
\p_p[I^{\bullet}([L]^3)] & \le \sum_{S\in\mathcal B_{\delta p^{-s}\log \frac{1}{p},\,  p^{-s-1/2}}}e^{-\Omega(p^{-s}(\log\frac 1p)^2)}
 \le L^4(3e)^{ p^{-s-1/2}}e^{-\Omega(p^{-s}(\log\frac 1p)^2)}
\to 0,
\end{align*}
for $\gamma>0$ small enough, and we are done.
\end{proof}


\section{Proof of Proposition \ref{twoout4}}\label{boundsalfa_s}
 Now, we proceed to prove all upper bounds of Proposition \ref{twoout4} in increasing order of difficulty. Most of the proofs are analogous versions of the cases $c<a+b$ in \eqref{myP1}, thus, we will sketch some of them and only point out what are the new ideas.

 Continuing with Definition \ref{seqpattern},
given $s\ge 3$ and $t=t_s$ given by \eqref{valueoft}, we also define
\begin{equation}
\alpha_s:=  \frac{
s+(s-1)+\cdots +(s-t)}{t+2}.
\end{equation}
\end{defi}
Note that inequality \eqref{valueoft} holds for $t=1$, since $s+(s-1)\le 3(s-1)$, thus, $t_s$ is well defined for all $s\ge 3$; moreover, it is easy to see that
\[t_s:=\max\{t\in[s-1]: \alpha_s< s-t\}=\left\lceil\frac{\sqrt{9+8s}-5}{2}\right\rceil.\]
Finally, we replace to obtain
$\alpha_s=
\frac{(t_s+1)s-(t_s+1)t_s/2}{t_s+2}
,$
which is the same as (\ref{alfas}). 

\begin{rema}
By the maximality of $t=t_s$, it follows that 
$s-t-1\le \alpha_s< s-t.$
Thus, $\alpha_s$ is integer if and only if $\alpha_s=s-t-1$, which occurs if and only if
\begin{equation}\label{alfa_sinteger}
    s=\frac{(t+1)(t+4)}{2}.
\end{equation}
\end{rema}

\subsection{Cases in Proposition \ref{twoout4} (i)}

Since the cases $c\ge a+b$ are covered by Theorem \ref{extendBeam}, from now on we assume $c<a+b$,
and set

\begin{equation}
    s:=r-c\in\{3,\dots, a\}.
\end{equation}

\subsubsection*{Case $c=a+b-s+m$ and $m\in [s-1]$}
In this section we consider the families \[\mathcal N_{a+b+m}^{a,b,a+b-s+m},\]
corresponding to the case  $r=c+s=a+b+m$ with $m\in[s-1]$. We are only considering the cases $r\le a+c$, thus, we assume $a\ge s.$ Set
\[M:=\max\{\alpha_s,m\}.\]
We will show the following.
\begin{prop}\label{upperweird>2}
	Consider $\mathcal N_{a+b+m}^{a,b,a+b-s+m}$-bootstrap percolation. There exists a constant $\Gamma=\Gamma(b)>0$ 
	such that, if 
	\begin{equation}
	L=\exp\left(\Gamma p^{-M}(\log \tfrac 1p)^{2}\right),
\end{equation}
then $ \p_p\left(I^\bullet([L]^3)\right)\to 1,\ as\ p\to 0.$
\end{prop} 
By Lemma \ref{regimer>a+b} we have
\begin{equation}\label{weirder}
  \p_p(I^\bullet([h+1]^2\times[w])|I^\bullet([h]^2\times[w]))\ge
	\left(1-e^{-\Omega(p^{b+m}w)}\right)^{c}
	\left(1-e^{-\Omega(p^mw)}\right)^{2h}  
\end{equation}

Now, we show our candidate for critical droplet.
\begin{lemma}
Set $L=\exp\left(\Gamma p^{-M}(\log \frac 1p)^{2}\right)$, $h=C_sp^{-\alpha_s}(\log \frac 1p)^{1/(t+2)}$ for some large constant $C_s$, and
\[R_1:=[h]^2\times 
\left[C_s'p^{-m}\log \tfrac{1}{p}\right].\]
Conditionally on $I^\bullet(R_1)$,
the probability that $[L]^3$ is internally filled goes to $1$ as $p\to 0$.
\end{lemma}
\begin{proof}
	Consider the rectangles $R_2\subset R_3\subset R_4\subset R_5:= [L]^3$ containing $R_1$, defined by
	$R_2:=[h]^2\times [p^{-(\overline{r}+1)}]$, $R_3:=[p^{-2s}]^2\times [p^{-(\overline{r}+1)}]$, 
	with $\overline{r}:=b+m$, and $R_4:=[p^{-2s}]^2\times [L]$. 
	The rest of the proof is as always;
	we apply Corollary \ref{1corols>2} to deduce that
	
		\begin{align*}
	\p_p(I^\bullet(R_{2})|I^\bullet(R_1)) & \ge \left(1-e^{-\Omega(p^{\alpha_s(t+2)}h^{t+2})}\right)^{p^{-(\overline{r}+1)}} 
	= 
	\left(1-p^{C_s}\right)^{p^{-(\overline{r}+1)}}\to 1,
	\end{align*}
	if $C_s$ is large. The rest of the proof is straightforward by using (\ref{weirder}).
	\end{proof}

Now, we prove the upper bound for the critical length.  Note that $M<2\alpha_s$.
\begin{proof}[Proof of Proposition \ref{upperweird>2}]
Set $L=\exp\left(\Gamma p^{-M}(\log \tfrac 1p)^{2}\right)$, 
where $\Gamma$ is a constant to be chosen. Set $w:=m(2\alpha_s-M)p^{-m}\log \tfrac{1}{p}$ and consider the rectangle
	\[R:=\left[C_sp^{-\alpha_s}(\log \tfrac 1p)^{1/(t+2)}\right]^2\times \left[w\right]\subset[L]^3.\]
	We need to show that there exists $C'>0$ satisfying
	\begin{equation}\label{spcdweird>2}
	\p_p(I^\bullet(R))\ge \exp\left(-C' p^{-M}(\log \tfrac 1p)^{2}\right).
	\end{equation}
In fact, start with $R_c:=[c]^2\times \left[
w\right]
\subset A$, and then grow from $R_k=[k]^2\times \left[w 
\right]$ to $R_{k+1}$, for 
\[k=c,\dots, K:=C_sp^{-\alpha_s}(\log \tfrac 1p)^{1/(t+2)}\]
to obtain

\begin{align*}
   \p_p\left(I^\bullet(R)\right) & \ge
   \p_p(R_c\subset A)\prod_{k=c}^K\p_p\left(I^\bullet(R_{k+1})|I^\bullet(R_k)\right)\\
   &\ge p^{|R_c|}\prod_{k=c}^K
		\left(1-e^{-\Omega(p^{b+m}w)}\right)^{c}
	\left(1-e^{-\Omega(p^mw)}\right)^{2h}   
	\ge p^{|R_c|}p^{C_1'K}\left(1-p^{2\alpha_s-M}\right)^{C_1''K^2} \\
	&\ge e^{-\Omega(w + K)\log\frac 1p} \exp\left(-\Omega(p^{2\alpha_s-M} p^{-2\alpha_s}(\log \tfrac 1p)^{2/(t+2)})\right)
	 \ge \exp\left(-C' p^{-M}(\log \tfrac 1p)^{2}\right),
\end{align*}
with all $C's$ being positive constants depending on $s$ and $c$.
\end{proof}

\subsection{Cases in Proposition \ref{twoout4} \textup{(ii)}}
These cases correspond to the regime
\[r\le a+b.\]
This time, for a given rectangle $R$, growing along the $e_1$ and $e_2$ (harder) directions is easier, because it is enough to find a copy of some `pattern' with constant size in all faces of $R$, no matter its size.

\begin{lemma}[Regime $r\le a+b$]\label{regimer<a+b}
	Fix $h,w\ge c$. If $p$ is small enough, then
\begin{equation}\label{growmidIneq}
    \p_p\left(I^\bullet([h+1]^2\times[w])|I^\bullet([h]^2\times[w])\right)\ge
	\left(1-e^{-\Omega\big(p^{{c\choose 2}}wh\big)}\right)^{2}
\end{equation}
\end{lemma}
\begin{proof}
See Lemma 2.4 in \cite{DB20}  (the induced 2-dimensional processes 
$\mathcal N_{r-a}^{b,c}$ and $\mathcal N_{r-b}^{a,c}$ are supercritical).
\end{proof}
We also know the cost of growing along the (easiest) $e_3$-direction.
\begin{coro}[Supercritical process]\label{1corols>2}
Consider $\mathcal N_s^{s,s}$-bootstrap percolation.
\begin{enumerate}
    \item[\textup{(a)}] Suppose that $p^{-(s-t-1)}\le k< \varepsilon 
    p^{-(s-t)}$ (in particular, if $p^{-\alpha_s}\le k< 
    p^{-\alpha_s-\delta}$ and $0<\delta<s-t-\alpha_s$), then 
\begin{equation}
\p_p(I^{\bullet}([k]^2)) \ge 
  1-\exp\left(-\Omega\left(p^{\alpha_s(t+2)}k^{t+2}\right)\right).
\end{equation}
    \item[\textup{(b)}] If $k\ge p^{-s}$, then
    \begin{equation}
\p_p(I^{\bullet}([k]^2)) \ge 
  1-\exp\left(-\Omega\left(k\right)\right).
\end{equation}
\end{enumerate}
\end{coro}
\begin{proof}
We use Lemma \ref{MainLemmafor>2} with $l=k$. For (a) take $m=s-t-1$   (recall that $s-t-1\le \alpha_s <s-t$). For (b) consider $m=s$.
\end{proof}

Now, we prove the three cases in (ii).

\subsubsection*{Case $c\in\{b+1,\dots, a+b-s\}$}
In this section we consider the families
\[\mathcal N_{c+s}^{a,b,c},\]
with $c\in I_s:=\{b+1,\dots, a+b-s\}$ (here $a>s$, otherwise this case does not exist).
We have to prove the following.
\begin{prop}\label{upper2s>2}
	Fix $c\in I_s$ and  consider $\mathcal N_{c+s}^{a,b,c}$-bootstrap percolation. There exists a constant $\Gamma=\Gamma(c)>0$
	such that, if 
	\begin{equation}
 L=\exp\left(\Gamma p^{-\alpha_s}(\log \tfrac 1p)^{(t+3)/(t+2)}\right),
\end{equation}
then $ \p_p\left(I^\bullet([L]^3)\right)\to 1,\ as\ p\to 0.$
\end{prop} 

The following result gives us the size of a critical droplet $R_1$.
\begin{lemma}\label{mimics>2}
Set $L=\exp\left(\Gamma p^{-\alpha_s}(\log \frac 1p)^{(t+3)/(t+2)}\right)$, 
$h=C_sp^{-\alpha_s}(\log \frac 1p)^{1/(t+2)}$ for some large constant $C_s$, and 
\[R_1:=[h]^2\times 
[c].\] Conditionally on $I^\bullet(R_1)$,
the probability that $[L]^3$ is internally filled goes to $1$, as $p\to 0$.
\end{lemma}
\begin{proof}
	Consider the rectangles $R_2\subset R_3\subset R_4\subset R_5:= [L]^3$ containing $R_1$, defined by
	\[R_2:=[h]^2\times [p^{-{c\choose 2}+\alpha_s-\delta}],\ \ \ \ \ R_3:=[p^{-s}]^2\times [p^{-{c\choose 2}+\alpha_s-\delta} 
	],\ \ \ \ \ R_4:=[p^{-s}]^2\times [L].\]
As usual,
	\[\p_p(I^\bullet([L]^3)|I^\bullet(R_1))\ge\prod_{k=1}^{4} \p_p(I^\bullet(R_{k+1})|I^\bullet(R_k)).\]
	By applying (\ref{growmidIneq}) and Corollary \ref{1corols>2} (a) we obtain
	
	\begin{align*}
	\p_p(I^\bullet(R_{2})|I^\bullet(R_1)) & \ge \left(1-e^{-\Omega\left(p^{\alpha_s(t+2)}h^{t+2}\right)}\right)^{p^{-{c\choose 2}+\alpha_s-\delta}} 
	 = \left(1-e^{-\Omega\left(C_s^{t+2}\log\frac 1p\right)}\right)^{p^{-{c\choose 2}+\alpha_s-\delta}} \\
	& \ge \left(1-p^{C_s}\right)^{p^{-{c\choose 2}+\alpha_s-\delta}} 
	\to 1,
	\end{align*}
if $C_s$ is large,	and
\begin{align*}
	\p_p(I^\bullet(R_{3})|I^\bullet(R_2)) & \ge 	\left(1-e^{-\Omega\big(p^{{c\choose 2}}\cdot p^{-{c\choose 2}+\alpha_s-\delta}
	\cdot h\big)}\right)^{2p^{-s}}
	\ge \left(1-e^{-p^{-\delta/2}}\right)^{2p^{-s}}
	\to 1,
	\end{align*}
and now we can use item (b) to get
	\begin{align*}
	\p_p(I^\bullet(R_{4})|I^\bullet(R_3)) &\ge \left(1-e^{-\Omega( p^{-s})}\right)^{L}
\to 1,
	\end{align*}
since $\alpha_s<s$. Finally, by  (\ref{growmidIneq}) it is clear that $\p_p(I^\bullet(R_{5})|I^\bullet(R_4))\to 1$.
\end{proof}

The proof of Proposition \ref{upper2s>2} is straightforward.

\begin{proof}[Proof of Proposition \ref{upper2s>2}]
	Set $L=\exp\left(\Gamma p^{-\alpha_s}(\log \frac 1p)^{(t+3)/(t+2)}\right)$, where $\Gamma$ is a constant to be chosen. Consider the rectangle
	\[R:=\left[C_sp^{-\alpha_s}(\log \tfrac 1p)^{1/(t+2)}\right]^2\times [c]\subset[L]^3.\]
	 As usual, it is enough to show that there exists a constant $C'>0$ such that
	\begin{equation}\label{spcd'>2}
	\p_p(I^\bullet(R))\ge \exp\left(-C' p^{-\alpha_s}(\log \tfrac 1p)^{(t+3)/(t+2)}\right),
	\end{equation}
We fill $R$ in the same way as before: start with $[c]^3\subset A$, and then grow from $R_k=[k]^2\times[c]$ to $R_{k+1}$, for $k=c,\dots, m:=C_sp^{-\alpha_s}(\log \tfrac 1p)^{1/(t+2)}$

\begin{align*}
   \p_p\left(I^\bullet(R)\right) 
   &\ge p^{c^3}\prod_{k=c}^m
	\left(1-e^{-\Omega(p^{{c\choose 2}}k)}\right)^{2} 
	\ge p^{c^3+c^2m}
	 \ge \exp\left(-C' p^{-\alpha_s}(\log \tfrac 1p)^{(t+3)/(t+2)}\right),
\end{align*}
for $C'>c^3$, and we are done.
\end{proof}

Finally, we deal with the cases that involve integration
 of the functions \[f_d(x)=-\log\big(1-e^{-x^d}\big).\]
 We first consider the isotropic case.
\subsubsection*{Case  $c=b=a$ and $r\in\{c+3,\dots, 2c\}$}
In this section, we consider $\mathcal N_{c+s}^{c,c,c}$-bootstrap percolation with $3\le s\le c$. 
We will prove the following.
\begin{prop}\label{upperccc>2}
	Consider $\mathcal N_{c+s}^{c,c,c}$-bootstrap percolation.  There exists a constant $\Gamma>0$
	such that, if
	\begin{equation}
	L=\exp\left(\Gamma p^{-\alpha_s}\right),
	\end{equation}
	then $	\p_p\left(I^\bullet([L]^3)\right)\to 1,\  as\ p\to 0.$
\end{prop} 
The induced process in all three directions is coupled by $\mathcal N_s^{s,s}$-bootstrap percolation, and recall that (Corollary \ref{1corols>2}) for $k\ge p^{-s}$
    \begin{equation}\label{intfill>last}
\p_p(I^{\bullet}([k]^2)) \ge 
  1-\exp\left(-\Omega\left(k\right)\right).
\end{equation}
This time we need to use the full strength of Lemma \ref{MainLemmafor>2}, which corresponds to all sizes $k\ge p^{-(s-t)}$.

\begin{coro}\label{2corol>2}
Consider $\mathcal N_s^{s,s}$-bootstrap percolation and $m\in \{s-t,\dots,s-1\}$. Suppose that $\varepsilon p^{-m}\le k< \varepsilon p^{-(m+1)}$, then 
\begin{equation}
\p_p(I^{\bullet}([k]^2)) \ge 
  1-\exp\left(-\Omega\left(p^{-\delta(s,m)/2}\right)\right),
\end{equation}
where $\delta(s,m):=-m^2+(2s+3)m-s(s+1)>0$.
\end{coro}
\begin{proof}
By setting $l=k$ and applying Lemma \ref{MainLemmafor>2} we get
\begin{align*}
\p_p(I^{\bullet}([k]^2)) & \ge 
1-\exp\left(-\Omega\left( p^{-sm+m^2+m}p^{s(s+1)/2-m(m+1)/2}\right)\right) =   1-\exp\left(-\Omega\left(p^{-\delta(s,m)/2}\right)\right).
\end{align*}
The inequality $\delta(s,m)>0$ is equivalent to
\[2s+3-\sqrt{9+8s}<2m,\]
which holds whenever $m\ge s-t$.
\end{proof}

Now, we can set the size of a rectangle that will grow w.h.p.
\begin{lemma}
	Set $L=\exp(\Gamma p^{-\alpha_s})$
	and $R_1:=[\varepsilon p^{-(s-t)}]^3$. Conditionally on $I^\bullet(R_1)$,
	the probability that $[L]^3$ is internally filled goes to $1$ as $p\to 0$.
\end{lemma}
\begin{proof}
	By (\ref{intfill>last}) and Corollary \ref{2corol>2} we have
	\begin{align*}
	\p_p\left(I^\bullet([L]^3)|I^\bullet(R_1)\right) & \ge 
		  \left( 1-e^{-\Omega\left(p^{-s}\right)}\right)^L
		  \prod_{m=s-t}^{s-1} \prod_{h=\varepsilon p^{-m}}^{p^{-(m+1)}} \p_p(I^\bullet([h+1]^3)|I^\bullet([h]^3))
\\
	  & \ge 	  \exp\left(-2Le^{-\Omega\left(p^{-s}\right)}\right)
	  \prod_{m=s-t}^{s-1} \left( 1-e^{-\Omega(p^{-\delta})}\right)^{p^{-(m+1)}},
	\end{align*}
where $\delta=\delta(s,m)/2>0$, and every factor goes to 1, as $p\to 0$. 
\end{proof}
Now, we are ready to show the upper bound.
\begin{proof}[Proof of Proposition \ref{upperccc>2}]
	Set $L=\exp(\Gamma p^{-\alpha_s})$, where $\Gamma$ is a constant to be chosen. Consider the rectangle
	\[R:=[\varepsilon p^{-(s-t)}]^3 \subset[L]^3.\]
As before, we only need to show that there is a constant $C'>0$, such that
\begin{equation}
	\p_p(I^\bullet(R))\ge \exp(-C' p^{-\alpha_s}).
	\end{equation}
Recall that
\[s-t-1\le \alpha_s<s-t.\]
It is enough to consider the (hardest) case $s-t-1=\alpha_s$, since we can use the same idea to deduce the case $s-t-1<\alpha_s$ (indeed, fewer steps are needed).

Remarkably, when $\alpha_s=s-t-1$ we can apply Lemma \ref{MainLemmafor>2} one more time for $m=s-t-2$ to get the lower bound needed to obtain right exponents. More precisely, since $\alpha_s$ is integer, by \eqref{alfa_sinteger} we know that
\[s=\frac{(t+1)(t+4)}{2},\]
thus, under $\mathcal N_s^{s,s}$-bootstrap percolation, if $\varepsilon p^{-(s-t-2)}=K_1 \le k < K_2= \varepsilon p^{-(s-t-1)}$,

\begin{align*}
\p_p(I^{\bullet}([k]^2)) & \ge 
  1-\exp\left(-\Omega\left(k^{s-m+1}p^{\sum_{i=s-t-1}^si}\right)\right)
 = 1-\exp\left(-\Omega\left( k^{t+3}p^{\frac 12(t+1)(t+2)(t+3)}\right)\right)\\
& \ge   1-\exp\left(-\Omega\left( k^{t+3}p^{\alpha_s(t+3)}\right)\right).
\end{align*}
While, by Corollary \ref{1corols>2} we already computed the following matching ratio in the exponents: if $\varepsilon p^{-(s-t-1)}= K_2 \le k <K_3= \varepsilon p^{-(s-t)}$,
\[
\p_p(I^{\bullet}([k]^2)) 
 \ge  1-\exp\left(-\Omega\left( k^{t+2}p^{\alpha_s(t+2)}\right)\right).
\]
On the other hand, by the discussion we had in the heuristics, we know that the is a lower bound which holds for $all$ values of $k$, namely,

\begin{align*}
    \p_p(I^{\bullet}([k]^2)) &
 \ge 1-\exp\left(-\Omega(p^{|B_s|}k^2)\right) \ge p^{C_s},
\end{align*}
for some large constant $C_s>0$.
This implies for $R_1:=[p^{-(s-t-2)}]^3$  that
\[\p_p(I^\bullet(R_1))\ge \p_p([c]^3\subset A)\prod_{h=c}^{K_1} p^{3C_s}\ge p^{c^3+3C_sp^{-(s-t-2)}}.\]

Finally, by setting $R_2:=[p^{-(s-t-1)}]^3$ we obtain 
	\begin{align*}
	\p_p(I^\bullet(R)) & \ge \p_p(I^\bullet(R_1))\p_p(I^\bullet(R_2)|I^\bullet(R_1))\p_p(I^\bullet(R)|I^\bullet(R_2))\\
	& \ge \p_p(I^\bullet(R_1)) \prod_{k=K_1}^{K_2}\left(1-e^{-\Omega\left( k^{t+3}p^{\alpha_s(t+3)}\right)}\right)^3\prod_{k=K_2}^{K_3} \left(1-e^{-\Omega\left( k^{t+2}p^{\alpha_s(t+2)}\right)}\right)^3 \\
	& \ge  \p_p(I^\bullet(R_1))\exp\left(3
	\int_{0}^{\infty}\log(1-e^{-\Omega\left( z^{t+3}p^{\alpha_s(t+3)}\right)})\,dz\right)\\
	&  \hspace{6cm} \times  \exp\left(3
	\int_{0}^{\infty}\log(1-e^{-\Omega\left( z^{t+2}p^{\alpha_s(t+2)}\right)})\,dz\right)
	\end{align*}

	\begin{align*}
	\\
	&\ge
	e^{-p^{-\alpha_s+1/2}} \exp\left(\tilde{C}p^{-\alpha_s}
	\left(
	\int_{0}^{\infty}\log(1-e^{-y^{t+3}})\,dy
+
\int_{0}^{\infty}\log(1-e^{-y^{t+2}})\,dy
	\right)\right)\\
	& \ge \exp\left(-C'p^{-\alpha_s}\right),
	\end{align*}
	for some constants $\tilde{C},C'>0$. 
\end{proof}

\subsubsection*{Case $c=b>a$}
In this section we cover the last case $c=b>a$. Consider the families
\[\mathcal N_{c+s}^{a,c,c}.\]
The corresponding upper bound goes as follows.
\begin{prop}
	Consider $\mathcal N_{c+s}^{a,c,c}$-bootstrap percolation with $c>a$. There exists a constant $\Gamma>0$ such that, if 
	\begin{equation}
	L=\exp\left(\Gamma p^{-\alpha_s}(\log \tfrac{1}{p})^{(t+1)/(t+2)}\right),
\end{equation}
then  $\p_p(I^\bullet([L]^3))\to 1,\ as\ p\to 0.$
\end{prop}

The proof is a combination of all ideas we have already seen, thus, we will only sketch it.
\begin{proof}[Sketch of the proof]
By following the proof of Corollary \ref{2corol>2}, we can see that under $\mathcal N_s^{s,s}$-bootstrap percolation, if $m\ge s-t$, then there exists a constant $\gamma>0$ such that, for all $\varepsilon p^{-m}\le w< \varepsilon p^{-(m+1)}$ and $w^{1-\gamma}\le l\le w$,
\begin{equation}
\p_p(I^{\bullet}([l]\times[w])) \ge 
  1-\exp\left(-\Omega\left(p^{-\delta/2}\right)\right),
\end{equation}
where $\delta= \delta(s,m,\gamma):=-m^2+(2s+3-\gamma)m-s(s+1)>0$ (apply continuity as $\gamma\to 0$).
This implies that the rectangle
\[R=\left[p^{-\alpha_s}(\log \tfrac 1p)^{-1/(t+2)}\right]\times\left[\varepsilon p^{-(r-c)}\right]^2\]
is internally filled with probability at least

\begin{equation}
	\p_p(I^\bullet(R))\ge \exp(-C p^{-\alpha_s}(\log \tfrac{1}{p})^{(t+1)/(t+2)}),
	\end{equation}
and $R$ can grow with high probability.
\end{proof}

\section{Future work}
In dimension $d=3$, a problem which remains open is the determination of the threshold for $c+3\le r\le b+c$ and $c<a+b$, and we think that the order of the critical length shoud be given by Proposition \ref{twoout4} and Conjecture \ref{targetgeneralc>b}.
We believe that the techniques used in \cite{DB20} and here can be adapted to cover these cases 
(though significant technical obstacles remain).

For dimensions $d\ge 4$, it is an open problem to determine $ L_c(\mathcal N_r^{a_1,\dots,a_d},p)$ for all values $a_1\le \cdots \le a_d$ and all $r\in\{a_d+1,\dots,a_d+ a_{d-1}\}$.
\section*{Acknowledgements}
The author would like to thank Janko Gravner and Rob Morris for their stimulating conversations on this project, and their many invaluable suggestions. 





\bibliographystyle{plain}
\bibliography{References}
\end{document}